\documentclass[12pt]{amsart}
\newtheorem{theorem}{Theorem}[section]
\newtheorem{lemma}[theorem]{Lemma}

\newtheorem{proposition}[theorem]{Proposition}

\newtheorem{definition}[theorem]{Definition}

\newcommand{\ncom}{\newcommand}

\ncom{\rar}{\rightarrow}
\ncom{\lrar}{\longrightarrow}
\ncom{\ov}{\overline}
\ncom{\m}{\mbox}
\ncom{\sta}{\stackrel}
\ncom{\comx}{{\mathbb C}}
\ncom{\Z}{{\mathbb Z}}
\ncom{\Q}{{\mathbb Q}}
\ncom{\R}{{\mathbb R}}
\ncom{\G}{{\mathbb G}}
\ncom{\al}{\alpha}
\ncom{\p}{{\mathbb P}}
\ncom{\E}{{\mathbb E}}
\ncom{\N}{{\mathbb N}}
\ncom{\K}{{\mathbb K}}
\ncom{\Le}{{\mathbb L}}
\ncom{\A}{{\mathbb A}}
\ncom{\F}{{\mathbb F}}

\ncom{\f}{\frac}
\ncom{\cA}{{\mathcal A}}
\ncom{\cX}{{\mathcal X}}
\ncom{\cO}{{\mathcal O}}
\ncom{\cW}{{\mathcal W}}
\ncom{\cL}{{\mathcal L}}
\ncom{\cP}{{\mathcal P}}
\ncom{\cH}{{\mathcal H}}
\ncom{\cS}{{\mathcal S}}
\ncom{\cM}{{\mathcal M}}
\ncom{\cC}{{\mathcal C}}
\ncom{\cT}{{\mathcal T}}
\ncom{\cF}{{\mathcal F}}
\ncom{\cN}{{\mathcal N}}
\ncom{\cJ}{{\mathcal J}}
\ncom{\cV}{{\mathcal V}}
\ncom{\cZ}{{\mathcal Z}}
\ncom{\cU}{{\mathcal U}}
\ncom{\cSU}{{\mathcal S \mathcal U}}
\ncom{\cG}{{\mathcal G}}
\ncom{\cQ}{{\mathcal Q}}
\ncom{\cR}{{\mathcal R}}
\ncom{\eop}{{\hfill $\Box$}}

\begin{document}
\baselineskip=16pt

\title[]{Tautological ring of the moduli space of generalised parabolic line bundles on a curve}

\author[J. N. Iyer]{Jaya NN Iyer}

\address{The Institute of Mathematical Sciences, CIT
Campus, Taramani, Chennai 600113, India}

\email{jniyer@imsc.res.in}

\footnotetext{Mathematics Classification Number: 14C25, 14D05, 14D20, 14D21 }
\begin{abstract}
In this paper, we consider the \textit{tautological ring} containing the extended Brill-Noether algebraic classes on the normalization of the compactified Jacobian of a complex nodal projective curve (with one node). This smallest $\Q$-subalgebra of algebraic classes under algebraic equivalence, stable under extensions of the maps induced by multiplication maps, Pontrayagin product and Fourier transform, is shown to be generated by pullback of the Brill-Noether classes of the Jacobian of the normalized curve and some natural classes.

\end{abstract}
\maketitle

%%%%%%%%%%%%%%%%%%%%%%%%%%%%%%%%%%%%%%%%%%%%%%%%%%%%%%%%%%%%%%%%%%%%%%%
\section{Introduction}

%%%%%%%%%%%%%%%%%%%%%%%%%%%%%%%%%%%%%%%%%%%%%%%%%%%%%%%%%%%%%%%%%%%%%%5

Suppose $X$ is a connected smooth projective curve of genus $g$ and defined over the complex numbers.
Let $D=x_1+x_2$ be an effective divisor on the curve $X$, such that the point $x_1$ is different from the point $x_2$. We recall the notion of generalised parabolic line bundles, due to U. Bhosle \cite{Bhosle}.  A generalised parabolic bundle of rank one on $(X,D)$ consists of the data (\cite[p.187]{Bhosle}): 

1) $E$ is a line bundle on $X$ 

2) $F^1(E)_D\subset E_{x_1}\oplus E_{x_2}$ is a rank one subspace of the direct sum of the fibres of $E$ at $x_1$ and $x_2$.

Th moduli space $P$ of generalised parabolic line bundles on $X$ is a smooth projective variety of dimension $g+1$ and is in fact a $\p^1$-bundle over the Jacobian variety $Jac(X)$, see \cite[Proposition 2.2]{Bhosle}.

U. Bhosle and A.J. Parameswaran \cite{BhosleParam} have defined natural subvarieties $\widetilde{W}_{i}\subset P$, for
$i=1,2,...,g$. These are defined as the Brill-Noether loci, similar to the naturally defined subvarieties $W_i\subset Jac(X)$, and they have proved a Poincar\'e formula in terms of the classes $\widetilde{W}_i$ in the group of algebraic cycles modulo numerical equivalence.

In this paper, we would like to understand the Poincar\'e relations in the rational ring of algebraic classes of $P$ modulo algebraic equivalence. More precisely, consider the group $A^k(P)$ of algebraic cycles of codimension $k$ on $P$, modulo algebraic equivalence and denote $A^k(P)_\Q:=A^k(P)\otimes \Q$. Then the direct sum $A^*(P)_\Q:=\oplus_{k\geq 0}A^k(P)_\Q$ is a commutative ring with the intersection product.
The first question that arises is whether the classical Poincar\'e formula for the classes $\widetilde{W}_i$,  holds in the ring $A^*(P)_\Q$. 

In this context, for the Jacobian $Jac(X)$ of a smooth projective curve $X$ of genus $g$, a result of Ceresa \cite{Ceresa} says that the Poincar\'e formula does not hold in the ring $A^*(Jac(X))_\Q$. To study other relations between the classes $W_i$, A. Beauville \cite{Beauville2} considered the tautological ring $\cR\subset A^*(Jac(X))_\Q$. The ring $\cR$ is the smallest $\Q$-subalgebra
containing the classes $W_i$, $1\leq i\leq g$, and stable under the pullback maps $\textbf{n}^*$, pushforward maps
$\textbf{n}_*$ and closed under the Pontryagin product $*$. Here $\textbf{n}:Jac(X)\rar Jac(X)$ is multiplication by the integer $n$, for any $n$. The Pontryagin product $*$ is defined in \cite{Beauville}, and we recall it in \S \ref{prelim}. He proved that the ring $\cR$ is in fact generated by the classes $W_i$, for $1\leq i\leq g$. 

To obtain similar results on $P$, we note that the multiplication maps, Pontryagin product and Fourier transform can be extended on the ring $A^*(P)_\Q$.
We consider the tautological ring $\widetilde{\cR}\subset A^*(P)_\Q$ containing the classes $\widetilde{W}_i, S_y, c_1(\cO_P(1))$ and the smallest $\Q$-subalgebra stable under natural analogues of $\textbf{n}^*,\textbf{n}_*$ and the Pontryagin product $*$, see \S \ref{extPP}, \ref{extFT}. Here $S_y$ is a section of the $\p^1$-bundle $\pi:P\rar Jac(X)$, defined in \S \ref{specialsection}. We show

\begin{theorem}
The $\Q$-algebra $\widetilde{\cR}$ is generated by $\pi^{-1}{W_i}$ and by $c_1(\cO_P(1))$.
\end{theorem}

The proof is given in \S \ref{finaltheorem}, and depends on the structure of the classes $\widetilde{W_i}$ and action of the extended Pontryagin product and Fourier transform on the ring of algebraic classes.

{\Small
Acknowledgements: We thank A.J.Parameswaran for useful discussions on \cite{BhosleParam} in Dec 2009. Thanks are also due to the referee for the comments.
} 

%%%%%%%%%%%%%%%%%%%%%%%%%%%%%%%%%%%%%%%%%%%%%%%%%%%%%%%%%%%%%%%%%%%%%%%%%%%
\section{Preliminaries}\label{prelim}
%%%%%%%%%%%%%%%%%%%%%%%%%%%%%%%%%%%%%%%%%%%%%%%%%%%%%%%%%%%%%%%%%%%%%%%%%%%

In this section, we recall some of the basic properties of the group of algebraic classes on the Jacobian
of a smooth projective curve of genus $g$. 

Suppose $X$ is a smooth projective curve of genus $g$. It can be embedded into the Jacobian variety $Jac(X)$, fixing a base point on $X$. Using the group law on $Jac(X)$, we define certain natural subvarieties $W_i := X + X + \ldots + X (i $ times). These subvarieties are well-defined up to translation. The classical Poincare's formula gives
$$
W_i = \frac{1}{(g-i)!} \, \theta^{g-i}
$$
and we see that $W_{g-1}$ is the theta divisor $\theta$ on $Jac(X)$. 

A. Beauville \cite{Beauville2} studied the tautological subring $\cR \subset A^*(Jac(X))_\Q$ which is the smallest $\Q$-subalgebra containing the classes $W_i$, $1 \leq i \leq g$ and stable under the pullback maps $\textbf{n}^*$, pushforward maps $\textbf{n}_*$, closed under the intersection product and the Pontryagin product $*$. The group $A^*(Jac(X))_\Q$ is a $\Q$-vector space, and graded by the codimension of the cycle classes.
The addition (or multiplication) map $m: Jac(X)\times Jac(X)\rar Jac(X)$, $(x,y)\rar x+y$, induces the pushforward map $m_*$ on the group of cycles on $Jac(X)\times Jac(X)$. This helps us to define the Pontryagin product as below.

 The group $A^*(Jac(X)_\Q$ has two natural multiplication laws which are associative and commutative: namely the intersection product and the Pontryagin product defined by:
$$
x*y := m_*(p^*x \cdot q^*y). 
$$

Here $p^*,q^*: A^*(Jac(X))_\Q \rar A^*(Jac(X))_\Q \times A^*(Jac(X))_\Q$ are the maps induced by the  first and the second projections $Jac(X)\times Jac(X)\rar Jac(X)$, respectively. Note that $*$ is homogeneous of degree $-g$.

We will identify $Jac(X)$ with it's dual using the principal polarization on it. We define $W^{g-i}  \in A^{g-i}(Jac(X)) $ as the class of $W_i$ in $A^*(Jac(X))$.

\begin{lemma} 
There is a second graduation on $A^p(Jac(X)) =  \oplus_{s}A^p(Jac(X))_{(s)}$ such that
$$
\textbf{n}^*x = n^{2p-s}x \,\, , \,\, \textbf{n}_*x = n^{2g-2p+s}x
$$
for $x \in A^p(Jac(X))_{(s)}  $ and for all $ n \in \Z.$ 
Also, $A^p(Jac(X))_{(s)} \neq 0 $ only if $ s < p \leq g+s $. Both products are homogeneous with respect to the second graduation.
\end{lemma}
\begin{proof}
 See \cite[Proposition 1 and Proposition 4]{Beauville}.
\end{proof}

We recall some of the properties of the Fourier transform for algebraic cycles on the Jacobian $Jac(X)$.

Let 
\begin{equation}\label{Poincareclass}
\ell := p^*\theta + q^*\theta - m^*\theta \in A^1(Jac(X) \times Jac(X))
\end{equation}
 which is the class of \emph{Poincare line bundle $\cL$} on $Jac(X) \times Jac(X)$. Fourier transform $\cF: A^*(Jac(X)) \rightarrow A^*(Jac(X)) $ defined by $\cF x = q_*(p^*x \cdotp e^{\ell}) $ satisfies following properties:

\subsection{}\label{Property2.1} $\cF \circ \cF = (-1)^g(-1)^*$
\subsection{}\label{Property2.2} $\cF(x * y) = \cF x \cdotp \cF y$ and $\cF(x \cdotp y) = (-1)^g \cF x * \cF y$ 
\subsection{}\label{Beauvillefourier} $\cF A^p(X)_{(s)} = A^{g-p+s}(X)_{(s)}$ (see \cite[Proposition 1]{Beauville}).

The main theorem in \cite{Beauville2} is:

\begin{theorem}\label{Beauville} 
$\cR$ is the $\Q$-subalgebra of $A^*(Jac(X))_\Q$ generated by the algebraic classes $W^1, W^2, \ldots ,W^{g-1}$.
\end{theorem}

%%%%%%%%%%%%%%%%%%%%%%%%%%%%%%%%%%%%%%%%%%%%%%%%%%%%%%%%%%%%%%%%%%%%%%%%%%%%%%

%%%%%%%%%%%%%%%%%%%%%%%%%%%%%%%%%%%%%%%%%%%%%%%%%%%%%%%%%%%%%%%%%%%%%%%%%%%%%
\section{Algebraic cycles on the moduli space of generalised parabolic line bundles on a curve}
%%%%%%%%%%%%%%%%%%%%%%%%%%%%%%%%%%%%%%%%%%%%%%%%%%%%%%%%%%%%%%%%%%%%%%%%%%%%%% 
Suppose $X$ is a connected nodal curve, with a single node at $x\in X$. The Jacobian $Jac(X)$ of $X$ is a non-compact variety, namely an extension of the Jacobian of the normalized curve by $\comx^*$. Oda and Seshadri \cite{OdaSeshadri} defined a compactification $\overline{Jac(X)}$ of $Jac(X)$ by adding rank one torsion free sheaves on $X$. However, this is a singular variety and hence it is difficult to understand algebraic classes of naturally defined subvarieties. So it is convenient to look at good compactifications of $Jac(X)$ and try to define extensions suitably. For this purpose, we recall U.Bhosle's work \cite{Bhosle} on rank one Parabolic sheaves and subsequent relations. 

\subsection{Generalised Parabolic line bundles on a curve}

Suppose $X$ is a connected smooth projective curve of genus $g$. U. Bhosle \cite{Bhosle} defined the notion of a generalised parabolic bundle on a smooth curve. This is relevant to the study of torsion free rank one sheaves on  nodal curves, as we will see below.

Fix an effective divisor $D$ on $X$, such that the points are distinct. For the sake of simplicity, in this paper, we will assume that $D=x_1+x_2$, $x_1\neq x_2$.

A generalised parabolic line bundle on $(X,D)$ is the data:

1) $E$ is a line bundle on $X$.

2) $F^1(E)_D\subset E_{x_1}\oplus E_{x_2}$ is a rank  one vector subspace. Here $E_{x_i}$ denotes the fibre of $E$ at  the point $x_i$.

A generalised parabolic line bundle as above will be denoted by the pair $(E,F^1(E)_D)$.

We have the following result:
\begin{proposition}
 The moduli space of generalised parabolic line bundles on $(X,D)$, where $D=x_1+x_2$, $x_1\neq x_2$, is a smooth projective variety. It is in fact a $\p^1$-bundle over the Jacobian of the normalized curve.
\end{proposition}
\begin{proof}
 See \cite[Proposition 2.2]{Bhosle}.
\end{proof}

\subsection{Relationship with torsion free sheaves on a nodal curve}\label{specialsection}

Suppose $X$ is an irreducible nodal curve of arithmetic genus $g+1$. Denote the nodal point by $x\in X$. Consider the desingularisation $p:X'\rar X$ of the curve $X$. Let $p^{-1}(x)=\{y,z\}$, where $y,z\in X'$ are smooth points. Consider the effective divisor $D:=y+z$ on the smooth projective curve $X'$ of genus $g$.

Firstly, we note that the Jacobian variety $Jac(X)$ of $X$ is a smooth group variety. In fact, it can be expressed as a central extension:
$$
1\rar \comx^* \rar Jac(X) \sta{p^*}{\rar} Jac(X')\rar 0. 
$$
Here $p^*$ is the morphism defined by the pullback of line bundles via $p:X'\rar X$.
In particular, the group variety $Jac(X)$ is not a compact variety. Oda and Seshadri \cite{OdaSeshadri} defined a natural compactification $\ov{Jac(X)}$ of $Jac(X)$. The variety $\ov{Jac(X)}$ is constructed as the moduli space of torsion free rank one sheaves on the curve $X$. However, the variety $\ov{Jac(X)}$
is not a smooth variety.
It was noticed in \cite{Bhosle}, that there is a morphism
$$
h: P\rar \ov{Jac(X)}.
$$
It is defined as follows.
Given a generalised parabolic data $(E,F^1(E)_D))$ on $X'$, consider
the exact sequence of sheaves on $X$: 
$$
0\rar F \rar p_*(E) \sta{q}{\rar} p_*(\frac{E_{y}\oplus E_{z}}{F^1(E)_D}) \rar 0.
$$
The term at the right end of the exact sequence is a skyscraper sheaf supported at the nodal point $x$.
Then the sheaf $F$ is defined as the kernel of the natural restriction map, and it is a torsion free rank one sheaf on $X$.
Furthermore, $F$ is locally free as long as the map
$q$ is not the projection to either $E_{y}$ or $E_{z}$. In other words, $F$ is locally free if and only if
$F^1(E)_D\neq E_{z} \m{ or } E_{y}$, see \cite[Proposition 1.8]{Bhosle}, \cite[Lemma 2.2]{BhosleParam}.

\begin{proposition}
There is a natural inclusion of the Jacobian $J(X)\subset P$ and the morphism $h$ restricts to an isomorphism on $J(X)$.  
\end{proposition}
\begin{proof} See \cite[889]{BhosleParam}.
The inclusion $Jac(X)\sta{i}{\hookrightarrow} P$ is given as: given $L_0\in Jac(X)$,
consider the exact sequence,
$$
0\rar L_0\rar p_*p^*L_0\rar Q \rar 0.
$$
Here $Q$ is a rank one skyscraper sheaf supported at the node $x$.
Consider the kernel $F^1:=\,ker \{(p_*p^*L_0)_{x}=(p^*L_0)_y\oplus (p^*L_0)_z\rar Q\}$. Then $i(L_0):=(p^*L_0,F^1)$. 
By construction $h$ is an isomorphism over $Jac(X)$.

\end{proof}

Let $S_y\subset P$ (resp. $S_z\subset P$) denote the section which corresponds to the points parametrizing $(p^*L,p^*(L)_y)\in P$ (resp. $(p^*L,p^*(L)_z$)).

We now look for subvarieties of $P$, which are an analogue of the subvarieties $W_i$ of the Jacobian of a smooth projective curve. These are defined in \cite{BhosleParam} as follows.
 
Denote $X_0$ the smooth locus of the nodal curve $X$, with one node at $x\in X$.
Fix a basepoint $p\in X_0$.

Define the morphism $f_d$, for $1\leq d\leq g$:
$$
\begin{array}{ccc}
\mbox{Sym}^d(X_0) & \rar & P \\
(x_1,x_2,...,x_d)& \mapsto & \cO_{X}(x_1+x_2+...+x_d - d.p). \\
\end{array}
$$

Denote $\widetilde{W_d}$ the closure of the image of the morphism $f_d$, with the reduced scheme structure.

We will use the following decomposition from \cite{BhosleParam}, for the projection $\pi:P\rar Jac(X')$. Recall that $X'\rar X$ is the normalization of $X$ such that $y,z\in X'$ lie over $x\in X$.

\begin{lemma}\label{genceresa}
There is a decomposition of cycles in the Chow group $CH_{g-d}(P)$:
$$
\widetilde{W}_{g-d}\,=\, \pi^{-1}(W_{g-d}).S_{y} + \pi^{-1}W_{g-d-1}.
$$
\end{lemma}
\begin{proof}
See \cite[Lemma 3.5,p.892]{BhosleParam}.
\end{proof}

We now have good analogues of the naturally defined subvarieties $\{W_i\}$. Our next aim is to extend the basic operations $\textbf{n}_*,\textbf{n}^*$, the Pontryagin product and Fourier transform on the group of cycles on $P$. This is done in the next few subsections.

%%%%%%%%%%%%%%%%%%%%%%%%%%%%%%%%%%%%%%%%%%%%%%%%%%%%%%%%%%%%%%%%%%

\subsection{Extension of the natural operators $\textbf{n}^*,\,\textbf{n}_*$ on $A^*(P)$}

Given an integer $n\in \Z$, consider the multiplication map
$$
\textbf{n}:Jac(X)\rar Jac(X),\, x\mapsto n.x.
$$
This map induces a group homomorphism $\textbf{n}^*:A^k(Jac(X))_\Q\rar A^k(Jac(X))_\Q$. Since the variety $Jac(X)$ is not a compact variety, the pushforward map $\textbf{n}_*$ is not defined.
In this subsection, we would like to define natural extensions of the maps $\textbf{n}^*$ and the pushforward map $\textbf{n}_*$ on the group $A^*(P)_\Q$.

\begin{proposition}\label{extendedmult}
Given an integer $n\in \Z$, the multiplication map $\textbf{n}$ extends on $P$ and defines the pushforward $\textbf{n}_*$ and pullback map $\textbf{n}^*$ on $A^k(P)_\Q$.
\end{proposition}
\begin{proof}
We need to check that the map $\textbf{n}:Jac(X)\to Jac(X)$  extends
to $\textbf{n}:P\to P$. This can be checked fibre wise.
The natural multiplication map on the fibres $\comx^*$ of
$Jac(X)\to Jac(X')$ is the usual map $a\mapsto a^n$ on
$\comx^*$. This map obviously extends to $\p^1$
fixing the complementary points $0$ and $\infty$.
\end{proof}

Consider the eigenspace decomposition of the groups of algebraic cycles:
$$
A^k(P)_\Q=\bigoplus_{s}A^k(P)_{(s)},
$$
where
$$
A^k(P)_{(s)}:= \{\alpha\in A^k(P)_\Q: \textbf{n}_*\alpha = n^{2g-2k+s}.\alpha, \,\textbf{n}^* \alpha = n^{2k-s}.\alpha, \m{ for all } n\in \Z \}.
$$

Consider the projection $\pi:P\rar Jac(X')$.

\begin{lemma}
The group of algebraic cycles $A^k(P)_\Q$ can be expressed as:
\begin{equation}\label{PBformula}
A^k(P)_\Q\,=\, A^{k}(Jac(X'))_\Q\oplus H.A^{k-1}(Jac(X'))_\Q.
\end{equation}
Here $H:=c_1(\cO_P(1))$.
\end{lemma}
\begin{proof}
This is a consequence of the projective bundle formula \cite{Fulton} applied to the $\p^1$-bundle $\pi: P\rar Jac(X')$.
\end{proof}
  
\begin{lemma}\label{eigendecomp}
The pushforward map $\textbf{n}_*$ and the pullback map $\textbf{n}^*$ are compatible with the decomposition in \eqref{PBformula}.
In particular, the eigenspace $A^k(P)_{(s)}$ can be written as:
$$
A^k(P)_{(s)}=A^{k}(Jac(X'))_{(s)}\oplus H.A^{k-1}(Jac(X'))_{(s)}.
$$
\end{lemma}
\begin{proof}
Clear.
\end{proof}

%%%%%%%%%%%%%%%%%%%%%%%%%%%%%%%%%%%%%%%%%%%%%%%%%%%%%%%%%%%%%%%%%%%%%%%%%%%%%%%

\subsection{Pontryagin product on the ring $A^*(P)_\Q$}\label{extPP}

We first consider the multiplication (or the addition) map
$$
m:Jac(X)\times Jac(X) \rar Jac(X),\, (a,b)\mapsto a+b.
$$
We first note that the map $m$ does not extend on the compactification $P$ (otherwise $P$ will be a group variety, which is not the case).
Hence, we consider the rational map:
$$
m:P\times P\rar P.
$$

After suitable blow-ups, we can resolve the rational map $m$ to get a commutative diagram:
\begin{eqnarray}\label{blowmultiplication}
\tilde{P}  &  & \\
\downarrow \!f & \sta{\tilde{m}}{\searrow} & \\
P \times P & \sta{m}{\rar} & P
\end{eqnarray}

\begin{lemma}
There is a Pontryagin product $*$ on the ring of algebraic cycles $A^*(P)_\Q$:
$$
*:A^k(P)_\Q\times A^l(P)_\Q\rar A^{k+l-g}(P)_\Q.
$$
In particular the ring $A^*(P)_\Q$ has two products, the intersection product and the Pontryagin product.
\end{lemma}
\begin{proof}
Consider the above resolution of the rational map $m$.
Note that the map $f$ is a sequence of blow-ups.
In fact it needs only one blow up: one
needs to blow up only $S_y\times S_z$ and $S_z\times
S_y$ in $P\times P$. Here $S_y\subset P$ (resp. $S_z\subset P$) denotes the section which corresponds to the points parametrizing $(p^*L,p^*(L)_y)\in P$ (resp. $(p^*L,p^*(L)_z$)).

 Hence, by the blow-up formula \cite{Fulton}, we have
$$
A^k(\tilde{P})_\Q\,=\,A^k(P\times P)_\Q \oplus A^k(S).
$$
Here $S\subset \tilde{P}$ is a proper closed subvariety, and determined by the centre of blow-ups. 
Denote the two projections on $P\times P$ by $p_1$ and $p_2$.
We now define the Pontryagin product $*$ as follows:
$$
*\,:\, A^k(P)_\Q\times A^l(P)_\Q\rar A^{k+l-g}(P)_\Q
$$
$$
*(\alpha,\beta)= \tilde{m}_*(f^*(p_1^*\alpha.p_2^*\beta))\,\in A^{k+l-g}(P)_\Q
$$
This gives the Pontryagin product on the ring $A^*(P)_\Q$.
\end{proof}

\begin{lemma}\label{compatiblePP}
The Pontryagin product $*$ on $A^*(P)_\Q$ is compatible with the decomposition   \eqref{PBformula}.
\end{lemma}
\begin{proof}
We just need to note that the section $c_1(\cO_P(1))$ and  any fibre of the $\p^1$-bundle $P$ is preserved under 
multiplication $\tilde{m}$. Hence $\tilde{m}_*$ reduces to $m'_*$ on the decomposition \eqref{PBformula}. Here $m'_*$ is induced by the multiplication $m'$ on $Jac(X')$. 

\end{proof}

%%%%%%%%%%%%%%%%%%%%%%%%%%%%%%%%%%%%%%%%%%%%%%%%%%%%%%%%%%%%%%%%%%%%%%%%%

\subsection{Fourier transform on the ring $A^*(P)_\Q$}\label{extFT}

We would like to now define Fourier transform on $P$.
Recall that on the usual Jacobian $Jac(C)$ of a smooth projective curve $C$, the Fourier transform is defined via the first Chern class $c_1(\cP)$ of the Poincar\'e line bundle $\cP$, on the product
$Jac(C)\times Jac(C)$. The class $c_1(\cP)$ is $p^*\theta + q^*\theta - m^*\theta$, where $\theta\subset Jac(C)$ is the theta divisor, see \eqref{Poincareclass}.

On the variety $P$ with projection $\pi:P\rar Jac(X')$, we consider the extended theta divisor class (see \cite[Lemma 3.5, p.892]{BhosleParam}):
\begin{equation}\label{extendedthetaclass}
\widetilde{W}_{g}:= S_{y} + \pi^{-1}W_{g-1}.   
\end{equation}
Here $g+1$ is the arithmetic genus of $X$ and $g$ is the genus of the normalization $X'$.

We define the extended Poincar\'e class as follows:
\begin{equation}\label{extendedPoincareclass}
\widetilde{\ell}:= p^*\widetilde{W}_g +q^*\widetilde{W}_g-f_*\widetilde{m}^*\widetilde{W}_g.
\end{equation} 
Here $\widetilde{m}$ and $f$ are defined in \eqref{blowmultiplication}.

\begin{lemma}\label{extendedPoincare}
The extended Poincar\'e class is 
$$
\widetilde{\ell}\,=\,(\pi\times \pi)^*\ell.
$$   
\end{lemma}
\begin{proof}
We first note, using Lemma \ref{genceresa},
$$
\begin{array}{ccc}
\widetilde{\ell} & = &  p^*(S_y) +p^*\pi^{-1}W_{g-1} + q^*(S_y) + q^*\pi^{-1}W_{g-1} -            f_*\widetilde{m}^*S_y - f_*\widetilde{m}^*(\pi^{-1}W_{g-1})\\
 &=&  p^*S_y +q^* S_y - f_*\widetilde{m}^*S_y + (\pi\times \pi)^* \ell\\
 &=& p^*S_y  +q^*S_y -p^*S_y -q^*S_y + (\pi\times \pi)^*\ell \\
 &=& (\pi\times \pi)^*\ell.
\end{array}
$$

\end{proof}

\begin{definition}
The Fourier transform $\widetilde{F}$ on $A^*(P)_\Q$ is defined as:
$$
\widetilde{\cF}: A^*(P)_\Q \rightarrow A^*(P)_\Q, 
$$ 
for $x\in A^*(P)_\Q$, let $\widetilde{\cF} x = q_*(p^*x \cdotp e^{\widetilde\ell})$. 
Here $p,q: P\times P\rar P$ are the first and second projections respectively. 
\end{definition}

\begin{lemma}
 The Fourier transform $\widetilde{\cF}$ 
 satisfies following properties:

1) $\widetilde{\cF} \circ \widetilde{\cF} = (-1)^g(-1)^*$

2) $\widetilde{\cF}(x * y) = \widetilde{\cF} x \cdotp \widetilde{\cF} y$ and $\widetilde{\cF}(x \cdotp y) = (-1)^g \widetilde{\cF} x * \widetilde{\cF} y$.
\end{lemma} 

\begin{proof}
From Lemma \ref{extendedPoincare}, we note that $e^{\widetilde{\ell}}=(\pi\times \pi)^*e^\ell$ and $\widetilde{\cF}$ is defined by this correspondence cycle. Hence, using the decomposition \eqref{PBformula}, compatibility of Pontryagin product Lemma \ref{compatiblePP}, and Properties
\eqref{Property2.1}, \eqref{Property2.2}, the assertion follows.

\end{proof}

\begin{lemma}
We have
$$
\widetilde{\cF}(A^p(P)_\Q)_{(s)}\, =\, A^{g-p+s}(P)_\Q)_{(s)}
$$
\end{lemma}
\begin{proof}
Use the decomposition in Lemma \ref{eigendecomp}  and apply Proposition \ref{Beauvillefourier}. 

\end{proof}

\section{The tautological ring $\widetilde{R}$ of $P$}\label{finaltheorem}

As in \cite{Beauville2}, consider the \textit{tautological} subring $\widetilde{\cR} \subset A(P)_\Q$ which is the smallest $\Q$-subalgebra containing the classes $\widetilde{W}_i$, $1 \leq i \leq g$, $S_y$ and $c_1(\cO_P(1))$, and stable under the pullback maps $\textbf{n}^*$, pushforward maps $\textbf{n}_*$, closed under the intersection product and the Pontryagin product $*$.

\begin{theorem}
The $\Q$-algebra $\widetilde{\cR}$ is generated by the classes $\pi^{-1}W_i,\,1\leq i\leq g-1$, $S_y$ and $c_1(\cO_P(1))$.
\end{theorem}
\begin{proof}
We just need to note that $\widetilde{\cR}$ is generated by the $\Q$-subalgebra $\cR$ and the 
classes $S_y$, $c_1(\cO_P(1))$. Indeed, the maps $\textbf{n}_*$, $\textbf{n}^*$ induced by multiplication by $\textbf{n}$, preserve $\cR$ and the classes
$S_y$, $c_1(\cO_P(1))$, by Proposition \ref{extendedmult}. 
Similarly, it is now straightforward to check that the Pontryagin product and Fourier transform preserve the $\Q$-subalgebra $<\cR, S_y,c_1(\cO_P(1))>$.
By Theorem \ref{Beauville}, the assertion follows.
\end{proof}

%%%%%%%%%%%%%%%%%%%%%%%%%%%%%%%%%%%%%%%%%%%%%%%%%%%%%%%%%%%%

\end{document}